\documentclass[10pt]{article}
 \usepackage[usenames]{color}
\usepackage{amssymb}
\usepackage{multicol}
\usepackage{graphicx}
\usepackage{amsthm}
\usepackage{verbatim}
\usepackage{float}
\usepackage{forest,adjustbox}
\usetikzlibrary{arrows,decorations.pathmorphing,backgrounds,positioning,fit,petri,calc}
\unitlength=1.0mm  \linethickness{0.4pt}
\usepackage{amsmath, epsfig}
\newtheorem{theorem}{Theorem}[section]

\newtheorem{proposition}{Proposition}[section]

\newcommand{\ecc}{{\rm ecc}}

\renewcommand{\phi}{\varphi}

\newcommand{\ignore}[1]{}
\definecolor{red}{RGB}{249,0,0}
\theoremstyle{plain}
\newtheorem{thm}{Theorem}[section]
\newtheorem{lem}[thm]{Lemma}
\newtheorem{cor}[thm]{Corollary}

\newtheorem*{ques*}{Question}

\theoremstyle{definition}

\newtheorem{defn}[thm]{Definition}

\theoremstyle{remark}

\numberwithin{subclaim1}{claim1}

\title{
Proximity and Remoteness in Directed and Undirected Graphs}
\author{
Jiangdong Ai\thanks{Department of Computer Science. Royal Holloway University of London.  {\tt Jiangdong.Ai.2018@live.rhul.ac.uk}.} \and Stefanie Gerke\thanks{Department of Mathematics. Royal Holloway University of London.  {\tt stefanie.gerke@rhul.ac.uk}.} \and Gregory Gutin \thanks{Department of Computer Science. Royal Holloway University of London. {\tt g.gutin@rhul.ac.uk}.} \and Sonwabile Mafunda\thanks{Department of Mathematics and Applied Mathematics. University of Johannesburg.  {\tt smafunda@uj.ac.za}.  This author's research partially supported by a British Council, Newton's Research Link grant and a South African Department of Higher Education and Training, New Generation of Academics grant.}}
\begin{document}
 \maketitle 

\begin{center}{\bf Abstract}
\end{center}
Let $D$ be a strongly connected digraph. The average distance $\bar{\sigma}(v)$ of a vertex $v$ of $D$ is the arithmetic mean of the distances from $v$ to all other vertices of $D$. The remoteness $\rho(D)$ and proximity $\pi(D)$ of $D$ are the maximum and the minimum of the average distances of the vertices of $D$, respectively. We obtain sharp upper and lower bounds on $\pi(D)$ and $\rho(D)$ as a function of the order $n$ of $D$ and describe the extreme digraphs
for all the bounds. We also obtain such bounds for strong tournaments. We show that for a strong tournament $T$, we have $\pi(T)=\rho(T)$ if and only if $T$ is regular. Due to this result, one may conjecture that every strong digraph $D$ with 
$\pi(D)=\rho(D)$ is regular. We present an infinite family of non-regular strong digraphs $D$ such that $\pi(D)=\rho(D).$ We describe such a family for undirected graphs as well.\\
\\
\textit{Keywords}: average distance, strongly connected oriented graph, tournament, regular graph, regular digraph

\section{Introduction} 
In this paper we consider only finite  graphs and digraphs without loops or parallel edges and arcs. 
We shall assume the reader is familiar with the standard terminology on graphs and digraphs and refer to \cite{GG} for terminology not discussed here.

Let $D=(V(D),A(D))$ be a strongly connected (often shortened to ``strong'') digraph with $n\ge 2$ vertices 
($D$ is {\em strong} if for every pair $u,v$ of vertices in $D$ there are a path from $u$ to $v$ and a path from $v$ to $u$). 
 We call
$n=|V(D)|$ the {\em order} of $D$ and $m=|A(D)|$ the {\em size} of $D.$ The \emph{distance} $d(u,v)$ from vertex  $u$ to vertex $v$ in $D$ is the length of a shortest $(u,v)$-dipath in $D$. 
The \emph{distance} of a vertex  $u\in V(D)$ is defined as $\sigma(u)=\sum_{v\in V(D)}d(u,v)$,
and the \textit{average distance} of a vertex as
$\bar{\sigma}(u\emph)=\frac{\sigma(u)}{n-1}.$
The \emph{proximity} $\pi(D)$ and the \emph{remoteness} $\rho(D)$ are  $\min\{\bar{\sigma}(u)|\;u\in V(D)\}$ and $\max\{\bar{\sigma}(u)|\;u\in V(D)\}$, respectively. These parameters were introduced independently by Zelinka \cite{z} and Aouchiche and Hansen \cite{a} for undirected graphs and 
studied in several papers, see e.g. \cite{a,b,d,dan,da,z}. We could not find any literature pertaining to proximity and remoteness for directed graphs. In this paper, we extend the study of these two distance measures from undirected graphs to directed graphs. (However, we also present a new result on the topic for undirected graphs.)


For additional terminology and notation, see the end of this section.


There are several results in the literature on proximity and remoteness and the relationship between these two distance measures and between them and other graph parameters such as minimum degree. 
Zelinka \cite{z} and Aouchiche and Hansen \cite{a} independently showed that in a
connected graph $G$ of order $n$, the proximity and the remoteness are bounded from above as follows.
\begin{thm}\label{thmA}
Let $G$ be a connected graph of order $n\ge 2.$Then
\[ 1 \leq \pi(G)\leq\left\{ \begin{array}{cc} 
   \frac{n+1}{4}
        & \textrm{if $n$ is odd,} \\ 
        \\  
    \frac{n+1}{4} + \frac{1}{4(n-1)}
        & \textrm{if $n$ is even.}
      \end{array} \right. \] 
The lower bound holds with equality if and only if $G$ has a vertex of degree $n-1$. The upper bound holds with equality if and only if $G$ is a path or a cycle. Also,
$1 \leq \rho(G) \leq n/2.$ The lower bound holds with equality if and only if $G$ is a complete graph. The upper bound holds with equality if and only if $G$ is a path.
\end{thm}
In Section \ref{strong digraph}, we prove Theorem \ref{p_1}, which is an analog of Theorem \ref{thmA} for strong digraphs. Note that the upper bound for proximity in Theorems \ref{p_1} is different from the one
in Theorem \ref{thmA}.

The difference between remoteness and proximity
for graphs with just given order was bounded by Aouchiche and Hansen \cite{a} as follows.
\begin{thm}\label{thmB}
 Let $G$ be a connected graph on $n\geq 3$ vertices with remoteness $\rho(G)$ and proximity $\pi(G)$ Then
\[\rho(G)- \pi(G)\leq\left\{ \begin{array}{cc} 
   \frac{n-1}{4}
        & \textrm{if $n$ is odd,} \\ 
        \\  
    \frac{n-1}{4} + \frac{1}{4(n-1)}
        & \textrm{if $n$ is even.}
      \end{array} \right. \] 
Equality holds if and only if $G$ is a graph obtained from a path $P_{\lceil\frac{n}{2}\rceil}$ and any connected graph $H$ on $\lfloor\frac{n}{2}+1\rfloor$ vertices by a coalescence of an endpoint of the path with any vertex of $H$.
\end{thm}
In Section \ref{strong digraph}, we obtain Theorem \ref{thm:rho-pi}, which is an analog of Theorem \ref{thmB} for strong digraphs. Note that the upper bound on the difference between remoteness and proximity
given in Theorem \ref{thm:rho-pi} is different from that in Theorem \ref{thmB}. 
Theorem \ref{thm:rho-pi} describes all strong digraphs $D$ for which the upper bound on $\rho(D)-\pi(D)$ holds with equality. A trivial lower bound for $\rho(D)-\pi(D)$ is 0. Neither Theorem \ref{thmB} nor Theorem \ref{thm:rho-pi} provide characterizations on when   $\rho(H)=\pi(H),$ where $H$ is a graph and digraph, respectively. While it seems to be hard to characterize\footnote{It appears to be interesting  open questions to characterize such strong digraphs and connected graphs.} all digraphs $D$ with 
 the property $\rho(D)=\pi(D),$ we start research in this direction in Section \ref{characterisation}. We prove that a tournament satisfies such a property if and only if it is regular. To show this result we first prove an analog of  
 Theorem \ref{p_1} for strong tournaments, which may be of independent interest. One may conjecture that every strong digraph $D$ with $\rho(D)=\pi(D)$ is regular. However, we demonstrate that this is already false for bipartite tournaments. In Section \ref{und}, we show that this is also false for undirected graphs.
 
We conclude our paper in Section \ref{sec:disc} with a discussion about distance parameters for directed and undirected graphs. For more research on proximity and remoteness  and related distance parameters for undirected graphs, see, e.g., \cite{a,b,dan,da}.


\paragraph{Additional Terminology and Notation.} 
The {\em complement} $\overline{D}$ of a digraph $D$ is a digraph with $V(\overline{D})=V(D)$ and $A(\overline{D})=\{uv\mid u\neq v\in V(D), uv\not\in A(D)\}.$ A digraph $D$ is {\em complete} if its complement has no arcs. 
 If $v$ is a vertex in a digraph $D$, then the {\em out-neighbourhood} of $v$ is $N^+(v)=\{u\in V(D)\mid vu\in A(D)\}$ and {\em in-neighbourhood} of $v$ is $N^-(v)=\{u\in V(D)\mid uv\in A(D)\}$, and the out-degree $d_D^+(v)$ and in-degree $d_D^-(v)$ of $v$ are  $|N^+(v)|$ and $|N^-(v)|,$ respectively. We denote by $\Delta^+(D)$ and $\Delta^-(D)$ the maximum out-degree and maximum in-degree of $D$ respectively, similarly, by $\delta^+(D)$ and $\delta^-(D)$ the minimum out-degree and minimum in-degree of $D,$ respectively. The maximum semi-degree and the minimum semi-degree of $D$ are $\Delta^0(D) = \max\{\Delta^+(D), \Delta^-(D)\}$ and $\delta^0(D) = \min\{\delta^+(D), \delta^-(D)\}$ respectively. 
A digraph $D$ is  {\em regular} if $\delta^0(D)=\Delta^0(D)$.

The {\em eccentricity} $\ecc(u)$ of a vertex $u$ in a digraph $D$ is the distance from $u$ to a vertex farthest from $u$, the \emph{radius} ${\rm rad}(D)$ and \emph{diameter} ${\rm diam}(D)$ are the minimum and maximum of all eccentricities of $D,$ respectively. For a nonnegative integer $p$, let $[p]=\{i\in \mathbb{Z}\mid 1\le i\le p\}.$ Thus, $[0]=\emptyset.$ 

If $i$ is a nonnegative integer and $v\in V(D)$, then $N^{+i}(v)$ is the set of all vertices at distance $i$ from $v$, and $n_i$ its cardinality. In particular, $N^{+0}(v)=\{v\}$ and $N^{+1}(v)=N^+(v).$
Observe that we have that $n_i \geq 1$ if and only if $0\leq i\leq {\rm ecc}(v)$. Let $p$ be a positive integer. A vertex $u$ of a digraph $D$ is called a {\em $p$-king} if $V(D)=\bigcup_{i=0}^p N^{+i}(u).$

We make use of the notion of {\em distance degree} $X(v)$ of $v$ defined in \cite{da}
as the sequence $(n_0,n_1,\ldots,n_{\ecc(v)}).$  For  an arbitrary sequence $X=(x_0,x_1,\dots ,x_{\ell})$ of integers, let $g(X)=\sum_{i=0}^{\ell}ix_i.$
Observe that $\sigma(v)=g(X(v)).$

An {\em orientation} of a graph $G$ is a digraph obtained from $G$ by replacing each edge by exactly one of the two possible arcs. We call an orientation of a complete graph and complete $k$-partite graph a {\em tournament} and a {\em $k$-partite tournament}, respectively. A 2-partite tournament is also called a {\em bipartite tournament}.

\section{Bounds on Proximity and Remoteness for Strong Digraphs}\label{strong digraph}


\begin{theorem}\label{p_1}
Let $D$ be a strong digraph on $n\geq 3$ vertices. Then 
\[1\leq \pi(D)\leq \frac{n}{2}.\] 
The lower bound holds with equality if and only if $D$ has a vertex of out-degree $n-1$. The upper bound holds with equality if and only if $D$ is a dicycle.
And,
\[1\leq \rho(D)\leq \frac{n}{2}.\]
The lower bound holds with equality if and only if $D$ is a complete digraph. The upper bound holds with equality if and only if $D$ is strong and contains a Hamiltonian dipath $v_1v_2\dots v_n$ such that 
 $\{v_iv_j\mid 2\le i+1<j\le n\}\subseteq A(\overline{D}).$ 
\end{theorem}
\begin{proof}
The lower bounds and extreme digraphs for them are trivial. For the upper bounds, since $\pi(D)\leq \rho(D)$, it suffices to prove only the upper bound for remoteness. Let $u$ be an arbitrary vertex of $D.$ It suffices to show that
$\sigma(u)\le {n \choose 2}.$ Let $X(u)=(n_0,n_1,\ldots,n_d)$ and let $v\in N^{+d}(u).$ Then there is a dipath from $u$ to $v$ implying that $n_i\ge 1$ for every $i\in [d]$ (clearly $n_0=1$). 
Let $X=X(u).$ Observe that $d\le n-1$ and if $d=n-1$ then $n_i=1$ for each $i\in [d]$  implying that $\sigma(u)=g(X)={n \choose 2}.$ If $d<n-1$ then there is a maximum $i$ such that $n_i>1.$ Observe that for $X'=(n_0,\dots, n_{i-1},n_i-1,n_{i+1},\dots ,n_d,1),$  we have $g(X)<g(X').$ Replacing $X$ by $X'$ and constructing a new $X'$ if $d$ is still smaller than $n-1,$ we will end up with $X=(1,\dots ,1)$ implying that $\sigma(u)< {n \choose 2}.$

Observe that the upper bound of $\pi(D)$ is reached if $D$ is a dicycle. Now let $D$ be a strong digraph such that is not $\overrightarrow{C}_{n}$ and let $u$ be a vertex such that $d=d^+_D(u)=\Delta^+(D)\geq 2$.  Then 
$(n-1)\pi(D)\le \sigma(u)=g(X(u)).$ Observe that $X(u)=(1,d,n_2,\dots ,n_{\ecc(u)})$ and $\ecc(u)\le n-d.$ Using the same argument as in the paragraph above, we conclude that $g(X(u))<g(Y),$ where $Y=(1,y_1,\dots ,y_{n-1})$ 
and $y_i=1$ for each $i\in [n-1].$ Thus, $(n-1)\pi(D)<{n \choose 2}$ and $\pi(D)<n/2.$ 

The upper bound of $\rho(D)$ is reached if and only if $D$ is strong and there is a vertex $v_1$ such that $\ecc(v_1)=n-1.$ The last condition is equivalent to 
$D$ containing a Hamiltonian dipath $v_1v_2\dots v_n$ such that $\{v_iv_j\mid 2\le i+1<j\le n\}\subseteq A(\overline{D})$ (as otherwise $\ecc(v_1)<n-1$).
\end{proof}



Theorem \ref{p_1} allows us to easily prove the following:

\begin{theorem}\label{thm:rho-pi}
Let $D$ be a strong digraph, then $\rho(D)-\pi(D)\leq \frac{n}{2}-1$. The upper bound holds with equality if and only if $D$ contains a Hamiltonian dipath $v_1v_2\dots v_n$ such that 
 $\{v_iv_j\mid 2\le i+1<j\le n\}\subseteq A(\overline{D})$ and there is at least one of the vertices $v_{n-1}, v_n$ has out-degree $n-1$. 
\end{theorem}

\begin{proof}
The upper bound follows from the lower bound on $\pi(D)$ and upper bound on $\rho(D)$ in Theorem \ref{p_1}. 
Observe that if $\rho(D)-\pi(D)=\frac{n}{2}-1$, then $\rho(D)=\frac{n}{2}$ and $\pi(D)=1$. By Theorem \ref{p_1}, this is equivalent to 
$D$ containing a Hamiltonian dipath $v_1v_2\dots v_n$ such that 
 $\{v_iv_j\mid 2\le i+1<j\le n\}\subseteq A(\overline{D})$ and the out-degree of at least  one of the vertices $v_{n-1}, v_n$ being $n-1.$
\end{proof}

\section{Digraphs $D$ with $\rho(D)=\pi(D)$} \label{characterisation}

We call a tournament $T$ \textit {regular} if $d^+(u)=d^-(u)=\frac{n-1}{2}$ for every $u\in V(T).$ This can happen only when $n$ is odd. 
A tournament is \textit {almost regular} if $n$ is even and $d^+(u)=\frac{n}{2}$ or $d^+(u)=\frac{n-2}{2}$ for every $u\in V(T).$

This section has two parts: the first proves that for a strong tournament $T,$  $\rho(T)=\pi(T)$ if and only if $T$ is regular,
and the second provides an infinite family of strong bipartite tournaments $T$ for which  $\rho(T)\ne \pi(T).$

\subsection{Tournaments} 

We will use the following well-known result.

\begin{proposition}\label{prop2}\cite{l}
Every tournament contains a 2-king, moreover, every vertex with maximum out-degree is a 2-king.
\end{proposition}

The next theorem is an analog of Theorem \ref{p_1}  for tournaments and allows us to easily show that for a strong tournament $T,$  $\rho(T)=\pi(T)$ if and only if $T$ is regular.

\begin{theorem}\label{p_3}
Let $D$ be a strong tournament on $n$ vertices, then 
\[\frac{n}{n-1}\leq\pi(D)\leq\left\{ \begin{array}{cc} 
   \frac{3}{2}
        & \textrm{if $n$ is odd,} \\ 
        \\  
    \frac{3}{2}-\frac{1}{2(n-1)}
        & \textrm{if $n$ is even}
      \end{array} \right. \] 
 The lower bound holds with equality if and only if $\Delta^+(D)=n-2.$
      The upper bound holds with equality if and only if $D$ is a regular or
      almost regular tournament.
And,
   \[\frac{n}{2}\geq\rho(D)\geq\left\{ \begin{array}{cc} 
   \frac{3}{2}
        & \textrm{if $n$ is odd,} \\ 
        \\  
    \frac{3}{2}+\frac{1}{2(n-1)}
        & \textrm{if $n$ is even}
      \end{array} \right. \]
      The lower bound holds with equality if and only if $D$ is a regular or almost regular tournament. The upper bound holds with equality if and only if $D$ is isomorphic to the tournament 
      $T_n$ with $V(T_n)=\{v_1,\dots ,v_n\}$ and $A(T_n)=\{v_iv_{i+1}\mid i\in [n-1]\}\cup \{v_jv_i\mid 2\le i+1<j\le n\}.$
\end{theorem}
\begin{proof}
Since $D$ is strong, $\Delta^+(D)\le n-2.$ Thus, $\frac{n}{n-1}\leq\pi(D).$
We have $\pi(D)=\frac{n}{n-1}$ if and only if there are vertices $u,v,w$ such that $ux\in A(D)$ for all $x\in V(D)\setminus \{u,v\},$ $vu\in A(D)$ but $wv\in A(D),$
which if and only if   $\Delta^+(D)=n-2$ since by  Proposition \ref{prop2} a vertex of out-degree $n-2$ is a 2-king.

For the upper bound on $\pi(D)$, consider a vertex $u$ of $D$ such that $d^+(u)=\Delta^+(D)$. Then, by Proposition \ref{prop2}, $u$ is a 2-king of $D$.
Since the average out-degree is $\frac{n-1}{2}$, we have that
\[d^+(u)\geq\left\{ \begin{array}{cc} 
   \frac{n-1}{2}
        & \textrm{if $n$ is odd,} \\ 
        \\  
    \frac{n}{2}
        & \textrm{if $n$ is even}.
      \end{array} \right. \]
      Now, \begin{align*}
        \sigma(u)&=\sum_{i=0}^{d}in_i\\
        &=\Delta^+(D)+2\left(n-\Delta^+(D)-1\right)\\
        &\leq\left\{ \begin{array}{cc} 
   \frac{3}{2}(n-1)
        & \textrm{if $n$ is odd,} \\ 
        \\  
    \frac{3}{2}n-2
        & \textrm{if $n$ is even}.
      \end{array} \right.
      \end{align*}
  And taking the average completes the proof to this bound. Observe that the equality holds only when \[d^+(u)=\Delta^+(D)=\left\{ \begin{array}{cc} 
   \frac{n-1}{2}
        & \textrm{if $n$ is odd,} \\ 
        \\  
    \frac{n}{2}
        & \textrm{if $n$ is even}.
      \end{array} \right. \]\\
  The upper bound of $\rho(D)$ and the fact that there is only one tournament $D$ (up to isomorphism) with $\rho(D)=n/2$
  follow from  the corresponding bound in Theorem \ref{p_1} and the characterization of digraphs $H$ for which  $\rho(H)=n/2$. 
  
  For the lower bound of $\rho(D)$, consider a vertex $u$ of $D$ such that $d^+(u)=\delta^+(D)$. We have that
  \[d^+(u)\leq\left\{ \begin{array}{cc} 
   \frac{n-1}{2}
        & \textrm{if $n$ is odd,} \\ 
        \\  
    \frac{n-2}{2}
        & \textrm{if $n$ is even}.
      \end{array} \right. \]
 Now, \begin{align*}
        \sigma(u)&=\sum_{i=0}^{d}in_i\\
        &\geq\delta^+(D)+2\left(n-\delta^+(D)-1\right)\\
        &\geq\left\{ \begin{array}{cc} 
   \frac{3}{2}(n-1)
        & \textrm{if $n$ is odd,} \\ 
        \\  
    \frac{3}{2}n-1
        & \textrm{if $n$ is even}.
      \end{array} \right.
      \end{align*}
  And taking the average completes the proof to this bound. Observe that the equality holds only when 
  \[d^+(u)=\delta^+(D)=\left\{ \begin{array}{cc} 
   \frac{n-1}{2}
        & \textrm{if $n$ is odd,} \\ 
        \\  
    \frac{n-2}{2}
        & \textrm{if $n$ is even}.
  \end{array} \right.\]
  Thus the lower bound holds with equality if and only if $D$ is a regular or almost regular tournament. 
\end{proof}

Theorem  \ref{p_3} allows us to easily obtain such a characterization of strong tournaments $T$ with  $\rho(T)=\pi(T).$

\begin{theorem}\label{prop}
For any strong tournament $T$, we have $\rho(T)=\pi(T)$ if and only if $T$ is a strong regular tournament. 
\end{theorem}
\begin{proof}
Since $\pi(D)\le \rho(D)$ for every strong digraph and by the upper bounds on $\pi(H)$ and lower bounds for $\rho(H)$ for strong tournaments in Theorem  \ref{p_3}
(including the characterizations of when the bounds hold), we have that for a strong tournament $T$ of order $n,$ $\rho(T)=\pi(T)$ if and only if $T$ is regular.
\end{proof}

\subsection{Bipartite Tournaments}

 Due to Theorem \ref{prop}, one may conjecture that every digraph $D$ with $\rho(D)=\pi(D)$ is regular. However, in this subsection we will show that this is not true and such counterexamples can be found already among bipartite tournaments. To obtain such counterexamples, we will first study some properties of strong bipartite tournaments.

In what follows, let $T=T[A,B]$ be a strong bipartite tournament with partite sets $A$ and $B$ of sizes $n$ and $m$, respectively.

\begin{defn}
A bipartite tournament is called \emph{bad} if there exists a vertex, the out-neighborhood of which is a proper subset of the out-neighborhood of another vertex, otherwise it is called \emph{good}. 
\end{defn}

\begin{defn}
For a vertex $v$, we denote by $M(v)$ the set of vertices with the same out-neighborhood as $v$, i.e., $M(v)=\{u| N^+(u)=N^+(v)\}$; and we denote $|M(v)|$ by $\mu(u)$. 
\end{defn}

\begin{lem}\label{lem:notnice}
If $T=T[A,B]$ be a bad strong bipartite tournament, then $\pi(T)\ne \rho(T).$
\end{lem}
\begin{proof}
Let vertices $u$ and $v$ be such that $N^+(u)\subset N^+(v).$ Then there is a vertex $w\in N^+(v)\setminus N^+(u).$
Note that for every $x\in (A\cup B)\setminus \{u,v,w\},$ we have $d(v,x)\le d(u,x).$ Since $vwu$ is a path in $T,$ we have $2=d(v,u)\le d(u,v).$ Finally, $1=d(v,w)<d(u,w).$
Thus, $\sigma(v)<\sigma(u)$ implying $\pi(T)\ne \rho(T).$
\end{proof}

Thus, we may restrict ourselves only to good strong bipartite tournaments.
The next lemma follows from the main result in \cite{GG}: for every $k\ge 2,$ if a $k$-partite tournament $D$ has at most one vertex of in-degree zero, then every vertex of $D$ is a 4-king. 

\begin{lem}\label{lem:nice}
Let $T$ be a good strong bipartite tournament, then every vertex of $T$ is a $4$-king.
\end{lem}


\begin{lem}\label{lem:formulas}
Let $T=T[A,B]$ be good, $\rho(T)=\pi(T)$ and $u\ne v\in (T).$ If both $u$ and $v$ are in $A$ or in $B$, then
\begin{equation}\label{eq:A}
\mu(u)-d^+(u)=\mu(v)-d^+(v)
\end{equation}
If $v\in A$ and $u\in B$ then 
\begin{equation}\label{eq:AB}
    2(\mu(v)-d^+(v))+|B|=2(\mu(u)-d^+(u))+|A|
\end{equation}
\end{lem}
\begin{proof}
For every vertex $v\in A$, $\sigma(v)=\sum_{i=0}^{4}in_i=d^+(v)+2(n-\mu(v))+3(m-d^+(v))+4(\mu(v)-1)=2(\mu(v)-d^+(v))+2n+3m-4.$ Similarly, if $v\in B$ then 
$\sigma(v)=2(\mu(v)-d^+(v))+2m+3n-4.$ The results of the lemma follow from $\sigma(u)=\sigma(v)$ when $u$ and $v$ are from the same partite set of $T$ and when they are from different partite sets of $T.$
\end{proof}

Lemmas \ref{lem:nice} and \ref{lem:formulas} imply the following:

\begin{cor}\label{cor:iff}
For a strong bipartite tournament $T,$ we have $\rho(T)=\pi(T)$ if and only if $T$ is good and there is a constant $c$ such that for every $v\in A$ and $u\in B$,
$$ 2(\mu(v)-d^+(v))+m=2(\mu(u)-d^+(u))+n=c$$ In particular, for a strong bipartite tournament $T$ with $n=m$, we have
$\rho(T)=\pi(T)$ if and only if $T$ is good and $d^+(u)- \mu(u)$ is the same for every vertex $u$.
\end{cor}

To help us find bipartite tournaments $T$ with $\rho(T)=\pi(T),$ one can use the following:

\begin{cor}\label{cor:reg}
Let $T$ be a good bipartite tournament and $c$ a constant such that $\mu(x)=c$ for every $x\in V(T).$  Then  $\rho(T)=\pi(T)$ if and only if $d^+(v)=m/2$ and $d^+(u)=n/2$ for every $v\in A$ and $u\in B.$  
\end{cor}
\begin{proof}
Suppose that $\rho(T)=\pi(T).$ By (\ref{eq:A}) and  $\mu(x)=c$ for every $x\in V(T),$ there exist constants $c'$ and $c''$ such that $d^+(v)=c'$ and $d^+(u)=c''$ for every $v\in A$ and $u\in B.$  Since $T$ has $nm$ arcs, $c'n+c''m=nm.$ By (\ref{eq:AB}) and  $\mu(x)=c$ for every $x\in V(T),$ we have $m-2c'=n-2c''$ implying $c''=c'-\frac{m-n}{2}.$ After substitution of $c''$ by $c'-\frac{m-n}{2}$ in $c'n+c''m=nm$ and simplification, we obtain $c'=m/2.$ This and $m-2c'=n-2c''$ imply $c''=n/2.$ 

Now suppose that $d^+(v)=m/2$ and $d^+(u)=n/2$ for every $v\in A$ and $u\in B.$  Since $\mu(x)=c$ for every $x\in V(T)$ and $T$ is good, by Corollary \ref{cor:iff}, we have $\rho(T)=\pi(T).$
\end{proof}

The following result shows, in particular, that there are non-regular digraphs $D$ for which $\rho(D)=\pi(D).$

\begin{theorem}\label{prop1}
For both $|A|=|B|$ and $|A|\ne |B|$, there is an infinite number of bipartite tournaments  $T$ with $\rho(T)=\pi(T).$ 
\end{theorem}
\begin{proof}
For the case $|A|=|B|,$ we can use the following simple bipartite tournament $T$. Let $n$ be even and let $A_1\cup A_2$ and $B=B_1\cup B_2$ be partitions of $A$ and $B$ into subsets of size $n/2.$ Let $N^+(v_i)=B_i$ for every $v_i\in A_i$ for $i=1,2.$ By Corollary \ref{cor:reg},  $\rho(T)=\pi(T).$ 

For the case of $|A|\ne |B|,$ first consider the following bipartite tournament $T_1$.  Let $A=\{1,2,3,4\}, B=\{1',2',3',4',5',6'\}$, and $N^+(1)=\{1',2',3'\},N^+(2)=\{1',4',5'\}, N^+(3)=\{2',4',6'\},$ $N^+(4)=\{3',5',6'\}.$
Observe that $ d^+(v)=3$  and  $\mu(v)=1$ for every $ v\in A$ and $d^+(u)=2$ and $\mu(u)=1$ for every $u\in B.$ Thus, by Corollary  \ref{cor:reg}, $T_1$ has $\rho(T_1)=\pi(T_1).$ Now we obtain a new bipartite tournament $T_t$ from $T_1$ by replacing every vertex $x$ in $T$ by $t\ge 2$ vertices $x_1,\dots , x_t$ such that $x_iy_j\in A(T_t)$ if and only if $xy\in A(T).$ Let $A_t$ and $B_t$ be the partite sets of $T_t$, $t\ge 1.$
Observe that for every $t\ge 1,$ $T_t$ is good, $\mu(x)=t$ for every $x\in V(T_t),$ and $d^+(v)=|B_t|/2$ and $d^+(u)=|A_t|/2$ for every $v\in A$ and $u\in B.$ Thus, by Corollary \ref{cor:reg} $\rho(T)=\pi(T).$
\end{proof}

\section{Undirected Graphs $G$ with  $\rho(G)=\pi(G)$}\label{und}

Using computer search, we found two non-regular graphs $G$ for which  $\rho(G)=\pi(G).$ The graphs are depicted in Figures \ref{fig:1} and \ref{fig:2}. Using the first of the graphs, we will prove the following:

\begin{theorem}
There is an infinite number of non-regular graphs $G$ with $\rho(G)=\pi(G).$ 
\end{theorem}

\begin{proof}
Consider the following non-regular graph $G$ (see Fig. \ref{fig:1}) with $V(G)=\{v_i\mid i\in \{0\}\cup [8]\}$ and
$$E(G)=\{v_0v_3, v_0v_6, v_0v_7, v_1v_4, v_1v_6, v_1v_8, v_2v_5, v_2v_7, v_2v_8, v_3v_6, v_3v_7, v_4v_6, v_4v_8, v_5v_7, v_5v_8\}.$$
Observe that $X(v_0)=X(v_1)=\ldots=X(v_5)=(1, 3, 4, 1)$ and $X(v_6)=X(v_7)=X(v_8)=(1, 4, 2, 2).$ Thus,  
$\sigma(v_i)= 22$ for every $i\in \{0\}\cup [8]$ implying that $\rho(G)=\pi(G).$ Now obtain a new graph $G_t$ from $G$ by replacing every vertex $x$ in $G$ by $t\ge 2$ vertices $x_1,\dots , x_t$ such that $x_iy_j\in E(G_t)$ if and only if $xy\in E(G).$ Hence for any pair $x_i,x_j\in E(G_t)$ we have that $d(x_i,x_j)=2$. Thus, $\sigma(x_i)=22+2(t-1)$ for every vertex $x_i$ of $G_t$ implying that $\rho(G_t)=\pi(G_t).$
\end{proof}

\begin{figure}[htp]
\centering
\includegraphics[width = 4cm]{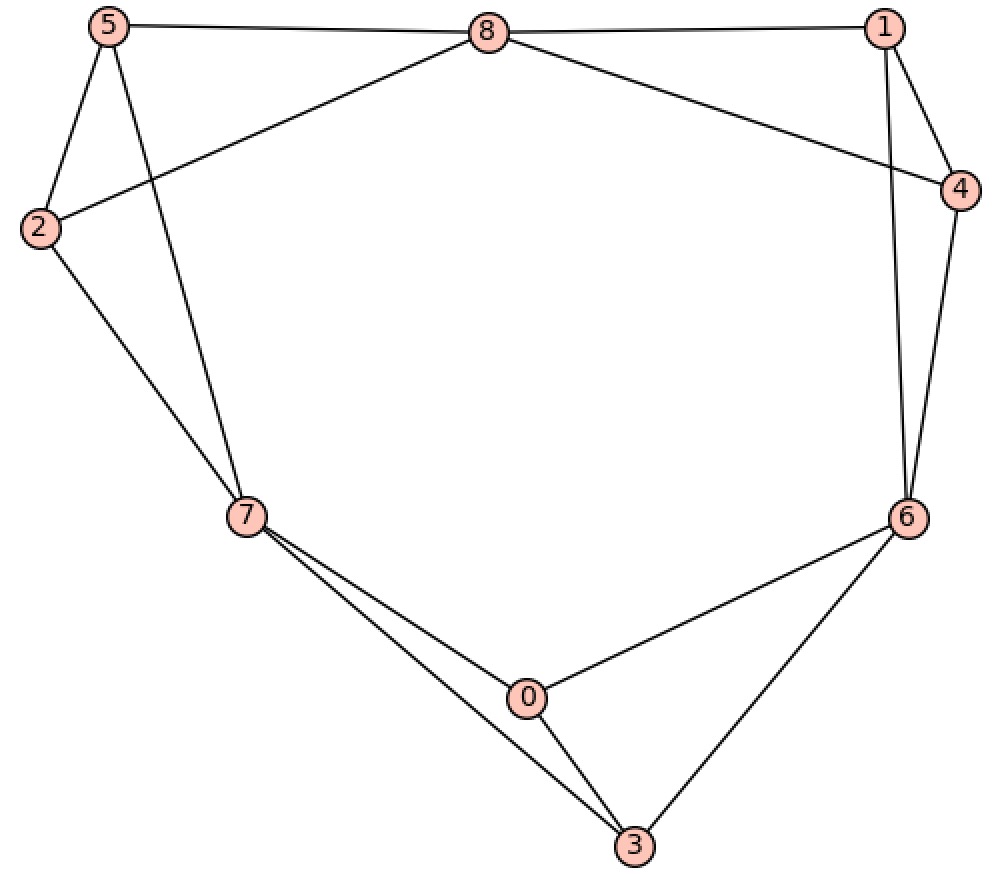}
\caption{A non-regular graph $G$ of order 9 with $\rho(G)=\pi(G)$}
\label{fig:1}
\end{figure}

\begin{figure}[htp]
\centering
\includegraphics[width = 6cm]{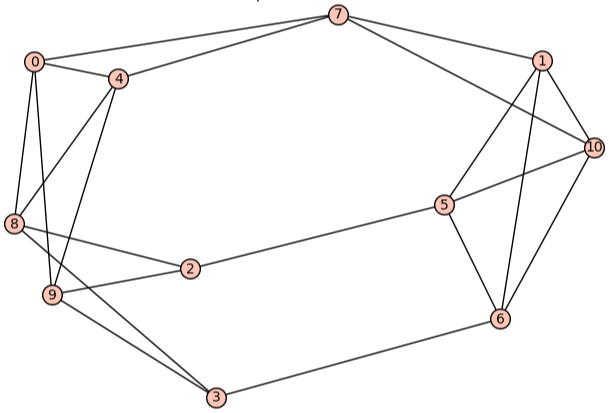}
\caption{A non-regular graph $G$ of order 11 with $\rho(G)=\pi(G)$}
\label{fig:2}
\end{figure}

\section{Discussion}\label{sec:disc}
The following inequalities hold by definition for both connected graphs and strong digraphs:
\[\pi(D)\leq\rho(D), \quad {\rm rad}(D)\leq {\rm diam}(D),\quad \pi(D)\leq {\rm rad}(D), \quad \rho(D)\leq {\rm diam}(D),\quad 1\leq {\rm diam}(D)\leq n-1.\]

For strong digraphs, the parameters ${\rm rad}(D)$ and $\rho(D)$ are incomparable as we can see from the following two examples. 
\begin{enumerate}
    \item[($i$)] Let $c\in [n-1]$ be arbitrary.  Denote by $D_c$ the strong digraph with vertices $v_0,v_1,\dots ,v_{n-1}$ and
    arc set $$A(D_c)=\{v_0v_j\mid j\in [n-1]\}\cup \{v_pv_{p+1}\mid p\in [n-2]\}\cup \{v_{n-1}v_1,v_cv_0\}.$$
 It is not hard to see that ${\rm rad}(D_c)=1$ and $\rho(D_c)=\frac{n}{2}$, hence, ${\rm rad}(D_c)< \rho(D)$ for $n\geq 3$.
    \item[($ii$)] If $D$ is a dicycle $\overrightarrow{C}_{n}$ on $n$ vertices, we have that ${\rm rad}(D)=n-1$ and $\rho(D)=\frac{n}{2}$, hence, ${\rm rad}(D)>\rho(D)$ for $n\geq 3$. 
\end{enumerate}

The following results hold for connected graphs but do not necessarily for strong digraphs.
\begin{enumerate}
\item[1.] For  a connected graph $G,$ ${\rm diam}(G)\leq 2{\rm rad}(G)$.
A simple example showing that this inequality does not hold for strong digraphs is when $D$ is the digraph $D_c$ defined above. Then ${\rm rad}(D)=1$ and ${\rm diam}(D)=n-1$, hence, ${\rm diam}(D)>2{\rm rad}(D)$.
\item[2.] For  a connected graph $G,$ if $v$ is a vertex of $G$ such that $\ecc(v)={\rm rad}(G)=r$, then for each $i\in [r-1]$, there exist at least two vertices at distance $i$ from $v$.
A counterexample to this observation for strong digraphs is when $D$ is a dicycle. Then $\ecc(v)={\rm rad}(G)=n-1$ for every vertex $v\in V(D),$ and for every vertex $v\in V(D)$ there is only one vertex of distance $i\in [n-2]$ from $v.$
\item[3.] For a connected graph $G$  of order $n$, if $n\geq 2$ then $1\leq {\rm rad}(G)\leq \lfloor\frac{n}{2}\rfloor$.
As counterexample for strong digraphs, again consider $D$ to be a dicycle on $n$ vertices. Indeed, ${\rm rad}(D)=n-1>\lfloor\frac{n}{2}\rfloor$ for $n\geq 3$.
\end{enumerate}

\end{document}